\documentclass[12pt]{amsart}
\usepackage{amssymb, colordvi}
\usepackage{multirow}
\usepackage{graphicx}

\numberwithin{equation}{section}
\newtheorem{theorem}{Theorem}
\numberwithin{theorem}{section}
\newtheorem{proposition}[theorem]{Proposition}

\newtheorem{ex}[theorem]{Example}
\newtheorem{rem}[theorem]{Remark}

\newtheorem{alg}[theorem]{Algorithm}

\newenvironment{example}{\begin{ex}\rm}{\end{ex}}
\newenvironment{remark}{\begin{rem}\rm}{\end{rem}}

\copyrightinfo{}{}

\newcounter{FNC}[page]
\def\fauxfootnote#1{{\addtocounter{FNC}{2}$^\fnsymbol{FNC}$%
     \let\thefootnote\relax\footnotetext{$^\fnsymbol{FNC}$#1}}}

\newcommand\Jac{\mbox{\rm J-det}}

\newcommand{\calA}{{\mathcal A}}
\newcommand{\calB}{{\mathcal B}}
\newcommand{\calH}{{\mathcal H}}
\newcommand{\calV}{{\mathcal V}}

\newcommand{\oD}{\overline{\Delta}}
\newcommand{\op}{\overline{p}}
\newcommand{\oq}{\overline{q}}
\newcommand{\oy}{\overline{y}}

\newcommand{\C}{\mathbb{C}}
\newcommand{\R}{\mathbb{R}}
\newcommand{\Z}{{\mathbb Z}}

\title{Khovanskii-Rolle continuation for real solutions} 

\author{Daniel J.~Bates}
\address{Department of Mathematics\\
         101 Weber Building\\
         Colorado State University\\
         Fort Collins, CO 80523-1874\\
         USA}

\email{bates@math.colostate.edu}
\urladdr{http://www.math.colostate.edu/~bates/}

\author{Frank Sottile}
\address{Department of Mathematics\\
         Texas A\&M University\\
         College Station\\
         Texas \ 77843\\
         USA}
\email{sottile@math.tamu.edu}
\urladdr{http://www.math.tamu.edu/\~{}sottile/}

\thanks{Bates and Sottile supported by the Institute for Mathematics and Its Applications.}
\thanks{Bates supported by NSF grant DMS-0914674.}
\thanks{Sottile supported by the NSF CAREER grant DMS-0538734 and NSF grant DMS-0701050.}  



\begin{document}

\begin{abstract}
 We present a new continuation algorithm to find all nondegenerate real 
 solutions to a system of polynomial equations.
 Unlike homotopy methods, it is not based on a deformation of the system; 
 instead, it traces
 real curves connecting the solutions of one system of equations to those 
 of another, eventually leading to the desired real solutions.
 It also differs from homotopy methods in that it follows only real paths 
 and computes no complex solutions of the original equations.
 The number of curves traced is bounded by the fewnomial bound for real
 solutions, and the method takes advantage of any slack in that bound.
\end{abstract}
\maketitle

%
\section*{Introduction}
%
Numerical continuation gives efficient algorithms 
%
%
for finding all complex solutions to a system of polynomial equations~\cite{SW05}.
Current implementations~\cite{BHSW06,HOM4PS,V99} routinely and reliably solve systems with
thousands of solutions having dozens of variables.
Often it is the  real solutions 
or the solutions with positive coordinates that are sought.
In these cases, numerical continuation first finds all complex solutions, which are then
sifted to obtain the desired real ones, i.e., those with sufficiently small imaginary part.
This is inefficient, particularly for fewnomial systems (systems of polynomials with
few monomials~\cite{Kh91}) which possess bounds on their numbers of positive~\cite{BS} and
of real~\cite{BBS} solutions that can be far smaller than their number of complex solutions.
More fundamentally, homotopy methods do not exploit the real algebraic nature of
the problem. 

We present a numerical continuation algorithm to find all nondegenerate real solutions 
(or all positive solutions) to a system of polynomials.
This \Blue{Khovanskii-Rolle continuation algorithm}
is modeled on the derivation of fewnomial
bounds~\cite{BBS,BS} for the numbers of real solutions.
It is efficient in that the number of
paths followed depends upon the corresponding fewnomial bounds and not on the number of
complex solutions. 
This algorithm also naturally takes advantage of any slack in the fewnomial bound.
All path tracing is along real curves and except for a precomputation, the
algorithm is completely real. 

To the best of our knowledge, this is the first continuation algorithm that finds only
real solutions to a system of equations.
There are other numerical algorithms that find only real solutions.
Exclusion~\cite[\S 6.1]{SW05} recursively subdivides a
bounded domain excluding subdomains that cannot harbor solutions.
This is an active area with research into better exclusion tests~\cite{Ge01,MP} and
complexity~\cite{AEG}.
Lasserre, Laurent, and Rostalski~\cite{LLR}  recently proposed another method based on
sums-of-square and semidefinite programming. 
Cylindrical algebraic decomposition~\cite{Co75} is a symbolic method that 
computes of a detailed stratification of the ambient space in addition to the solutions.
Gr\"obner bases may be used to compute a rational univariate representation~\cite{RUR}
from which the solutions are found by substituting the roots of a univariate
polynomial into the representation.

Many existing methods have drawbacks.
Exclusion suffers from the curse of dimensionality:  the growth in the number of cells may
be exponential in the number of variables.  
Cylindrical algebraic decomposition is infeasible in 
dimensions above 3 due to the complexity of the data structures involved.  
Gr\"obner basis algorithms essentially compute all complex solutions and much more
in addition to the real solutions.

Our algorithm is based on Khovanskii's generalization of Rolle's Theorem~\cite{Kh91} and
the proof of the fewnomial bound~\cite{BBS,BS}, and we are indebted to Bernd 
Sturmfels who observed that this proof leads to a numerical algorithm to compute real
solutions to systems of polynomials.
This algorithm does not directly solve the given polynomial system but rather an
equivalent (Gale dual~\cite{BS_Gale}) system consisting of master functions in the
complement of an arrangement of hyperplanes.
If the original system involves $n$ polynomials in $n$ variables having a total of 
$l{+}n{+}1$ monomials, then the Gale dual system consists of $l$ master functions in the 
complement of an arrangement of $l{+}n{+}1$ hyperplanes in $\R^l$.
Gale duality~\cite{BS_Gale} asserts that the two systems are equivalent and produces an
algebraic bijection between complex solutions, which restricts to a bijection between
their real solutions, and further restricts to a bijection between their (appropriately
defined) positive solutions. 

Here, we describe the Khovanskii-Rolle algorithm and a  proof-of-concept 
implementation.  This implementation uses the system Bertini~\cite{BHSW06} for 
precomputation, but it is otherwise implemented in Maple.
It is also restricted to the case of finding positive solutions when $l=2$.
Detailed examples are on our webpage~\cite{KRWeb}, which also contains our software and a
brief guide.
We plan significant further work on this algorithm,
including source code and compiled binaries of an implementation 
that places no restrictions on $l$ and links to numerical homotopy continuation
software~\cite{BHSW06,HOM4PS,V99} for the precomputation.
We also plan theoretical work on the technical challenges posed by this
algorithm.

As the algorithm does not solve the original system, but rather an equivalent system in a
different number of variables, we first explain that transformation, together with an
example and an important reduction in Section~\ref{S:Background}. 
We also discuss curve tracing and homotopy continuation in Section~\ref{S:Background}.  
In Section~\ref{S:Algorithm}, we explain our algorithm, give a proof of correctness,
and examine its complexity.
We illustrate the Khovanskii-Rolle algorithm on an example in Section~\ref{S:examples} 
and describe the performance of our implementation on this example and on one that is 
a bit more extreme.
Surprisingly, our simple maple implementation for finding positive solutions outperforms
both Bertini~\cite{BHSW06} and PHCpack~\cite{V99}.
In Section~\ref{S:Implementation}, we describe our implementation, and 
we end the paper with a discussion of future work.  Please note that 
we intend ``nondegenerate real solutions'' when we say ``real solutions'' 
throughout, unless otherwise noted.

%
\section{Background}\label{S:Background}
%

We first explain how Gale duality transforms a system of polynomial equations into an
equivalent system of rational master functions.
This reduces the problem of finding all positive solutions to the polynomial equations to finding
solutions to the master functions in a polyhedron, which we may assume is bounded.
Then we discuss general curve-tracing algorithms and
finally outline the method of homotopy continuation to highlight how it differs from
Khovanskii-Rolle continuation.

%
\subsection{Gale Duality}
%
The Khovanskii-Rolle algorithm does not in fact directly solve polynomial systems, but
rather solves systems of master functions defined in the complement of an arrangement
of hyperplanes.
Solutions of the system of master functions correspond to solutions of the original system
via Gale duality~\cite{BS_Gale}.
We begin with an example.

\begin{example}\label{Ex:Gale}
The system of Laurent polynomials (written in diagonal form)
%
 \begin{equation}\label{Eq:Ex_poly}
  \begin{array}{rclcrcl}
   cd            &=&             \frac{1}{2}be^2 + 2a^{-1}b^{-1}e-1  &\quad&
   cd^{-1}e^{-1} &=& \frac{1}{2}(1 + \frac{1}{4}be^2     - a^{-1}b^{-1}e)\\
   bc^{-1}e^{-2} &=& \frac{1}{4}(6 - \frac{1}{4}be^2     - 3a^{-1}b^{-1}e) 
    &\quad&\rule{0pt}{14pt}
   bc^{-2}e      &=& \frac{1}{2}(8 - \frac{3}{4}be^2   - 2a^{-1}b^{-1}e)\\
   ab^{-1}       &=&             3 - \frac{1}{2}be^2   + a^{-1}b^{-1}e\ ,\rule{0pt}{14pt}
  \end{array}  
 \end{equation}
%
%
%
has as support (set of exponent vectors) the columns of the matrix
\[
  \Blue{\calA}\ :=\ \left(
  \begin{array}{rrrrrrr}
    -1 & 0 & 0 &  0 &  0 &  0 &  1\\
    -1 & 0 & 1 &  0 &  1 &  1 & -1\\
     0 & 1 & 0 &  1 & -1 & -2 &  0\\
     0 & 1 & 0 & -1 &  0 &  0 &  0\\
     1 & 0 & 2 & -1 & -2 &  1 &  0
  \end{array}\right)\ .
\]
Since 
\[
  \calA\left(
  \begin{array}{rrrrrrr}
    -1& 1& -2& 1& -2& 2& -1\\
     1& 6& -3& 6& -2& 7&  1     
  \end{array}\right) ^T
  \ =\ \left(
  \begin{array}{rr} 0&0\\  0&0\\  0&0\\  0&0\\  0&0  \end{array}\right)\ ,
\]
we have the following identity on the monomials
 \begin{equation}\label{Eq:heptagon_monoms}
  \begin{array}{rcl}
    (a^{-1}b^{-1}e)^{-1}(cd)^1(be^2)^{-2}(cd^{-1}e^{-1})^1(bc^{-1}e^{-2})^{-2}
      (bc^{-2}e)^2(ab^{-1})^{-1}
      &=& 1\,,\\
    (a^{-1}b^{-1}e)^1\hspace{6.7pt}(cd)^6(be^2)^{-3}(cd^{-1}e^{-1})^6(bc^{-1}e^{-2})^{-2}
     (bc^{-2}e)^7(ab^{-1})^{1}\hspace{6.7pt}
     &=& 1\,.  \rule{0pt}{15pt}
  \end{array}
 \end{equation}

A Gale system dual to~\eqref{Eq:Ex_poly} is obtained from the
identity~\eqref{Eq:heptagon_monoms} by first substituting $x$ for the term
$\frac{1}{4}be^2$ and $y$ for $a^{-1}b^{-1}e$ in~\eqref{Eq:Ex_poly} to obtain 
\[
  \begin{array}{rclcrcl}
   cd            &=&             2x + 2y-1  &\quad&
   cd^{-1}e^{-1} &=& \frac{1}{2}(1 +  x - y)\\
   bc^{-1}e^{-2} &=& \frac{1}{4}(6 -  x - 3y) &\quad&\rule{0pt}{14pt}
   bc^{-2}e      &=& \frac{1}{2}(8 - 3x - 2y)\\
   ab^{-1}       &=&             3 - 2x  + y\ ,\rule{0pt}{14pt}
  \end{array} 
\]
Then, we substitute these linear polynomials for the monomials
in~\eqref{Eq:heptagon_monoms} to obtain the system of rational master functions
 \begin{equation}\label{Eq:first_master}
  \begin{array}{rcl}
    y^{-1}(2x{+}2y{-}1)\ (4x)^{-2}\Bigl(\frac{1 +  x - y}{2}\Bigr)\ 
    \Bigl(\frac{6 -  x - 3y}{4}\Bigr)^{\!-2}
    \Bigl(\frac{8 - 3x - 2y}{2}\Bigr)^{\!2}(3 {-} 2x {+} y)^{-1}
      &=& 1\,,\\
     y(2x{+}2y{-}1)^6(4x)^{-3}\Bigl(\frac{1 +  x - y}{2}\Bigr)^{\!6}
     \Bigl(\frac{6 -  x - 3y}{4}\Bigr)^{\!-2}
     \Bigl(\frac{8 - 3x - 2y}{2}\Bigr)^{\!7}(3 {-} 2x{+} y)\quad \rule{0pt}{20pt}
    &=& 1\,.
  \end{array}
 \end{equation}
 If we solve these for $0$, they become $\Brown{f}=\Blue{g}=0$, where 
 \begin{equation}\label{Eq:Mfg}
  \begin{array}{rcl}
   \Brown{f}&:=& 
    \Brown{ (2x{+}2y{-}1)(1{+}x{-}y)(8{-}3x{-}2y)^2
    \ -\ 8yx^2(6{-}x{-}3y)^2(3{-}2x{+}y)}\,,\\ 
   \Blue{g}&:=&
  \Blue{   y(2x{+}2y{-}1)^6(1{+}x{-}y)^6   (8{-}3x{-}2y)^7(3{-}2x{+}y)
    \ -\ 32768\, x^3(6{-}x{-}3y)^2 }\,.
\rule{0pt}{18pt} 
  \end{array}
 \end{equation}
%

Figure~\ref{F:Gale_pic1} shows the curves defined by $f$ and $g$ and the 
lines given by the linear factors in $\Brown{f}$ and $\Blue{g}$.
The curve $\Brown{f}=0$ has the three branches indicated and the other arcs belong to
$\Blue{g}=0$. 
\begin{figure}[htb]
 \[
   \begin{picture}(310,185)(-73,-20)
    \put(0,0){\includegraphics[height=160pt]{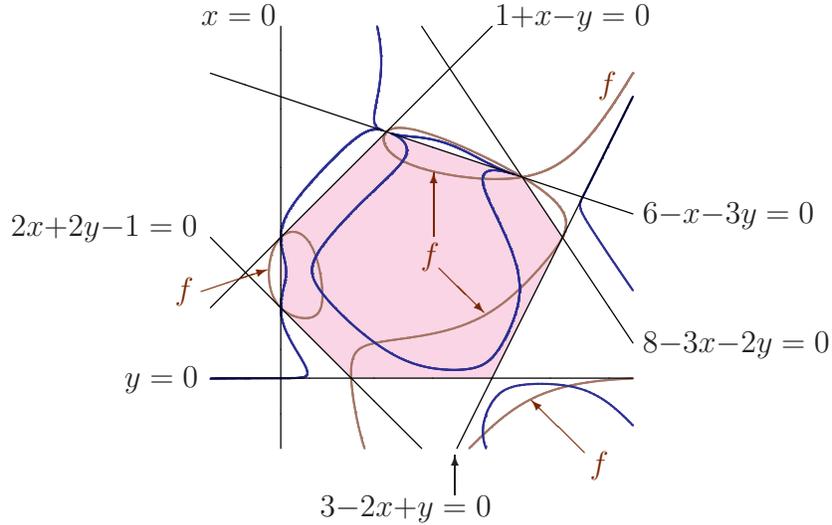}}
    \put(-13,57){$\Brown{f}$} \put(-3,60){\Brown{\vector(3,1){25}}}
    \put(80, 70){$\Brown{f}$}\put(87,69){\Brown{\vector(1,-1){17}}}
      \put(85,81){\Brown{\vector(0,1){24}}}
    \put(147, 135){$\Brown{f}$} 
     \put(144,-9){$\Brown{f}$} \put(142,-1){\Brown{\vector(-1,1){20}}}

\put(-32, 24){$y=0$}
\put(-75, 81){$2x{+}2y{-}1=0$}
\put( -3,161){$x=0$}
\put(108,161){$1{+}x{-}y=0$}
\put(164, 86){$6{-}x{-}3y=0$}
\put(164, 37){$8{-}3x{-}2y=0$}
\put( 42,-25){$3{-}2x{+}y=0$}\put(93,-17){\vector(0,1){15}}


   \end{picture}
 \]
\caption{Curves and lines}\label{F:Gale_pic1}
\end{figure}

It is clear that the solutions to $\Brown{f}=\Blue{g}=0$ in the complement of the lines
are consequences of solutions to~\eqref{Eq:Ex_poly}.
More, however, is true.
The two systems define isomorphic schemes as complex or as real varieties, with the
positive solutions to~\eqref{Eq:Ex_poly} corresponding to the solutions
of $\Brown{f}=\Blue{g}=0$ lying in the central heptagon.
Indeed, the polynomial system~\eqref{Eq:Ex_poly} has 102 solutions in $(\C^\times)^5$.
This may be verified with either symbolic software such as Singular~\cite{GPS05},
or with numerical solvers.
Ten of these solutions are real, with six positive.
We computed them using both Bertini~\cite{BHSW06} and PHCpack~\cite{V99} (which produce 
the same solutions) and display them (in coordinates $(a,b,c,d,e)$) to 3
significant digits. 
 \begin{equation}\label{Eq:Sols_abcde}
  \begin{array}{rl}
   (-8.92,   -1.97, -0.690,  3.98, -1.28)\,,&
   (-0.0311,  8.52, -1.26, -14.1,  -1.39)\,,\\
   (-0.945,   3.41,  1.40,   0.762, 1.30)\,,&\rule{0pt}{13pt}
   (-5.21,    4.57,  2.48,   1.20,  1.34)\,,\\
   ( 2.19,    0.652, 0.540,  2.27 , 1.24)\,,&\rule{0pt}{13pt}
   ( 2.45,    0.815, 0.576,  1.38,  1.20)\,,\\
   ( 3.13,    1.64,  0.874,  0.962, 1.28)\,,&\rule{0pt}{13pt}
   ( 2.00,    0.713, 1.17,   3.28,  2.20)\,,\\
   ( 1.61,    1.03,  2.37,   1.98,  2.35)\,,&\rule{0pt}{13pt}
   ( 0.752,   3.10,  2.36,   1.55,  1.48)\,.
  \end{array}
 \end{equation}
%
%

For the system  $\Brown{f}=\Blue{g}=0$ of master functions~\eqref{Eq:Mfg}
in the complement of the lines shown in Figure~\ref{F:Gale_pic1}, 
symbolic and numerical software finds 102 solutions with ten real solutions
and six lying in the heptagon of Figure~\ref{F:Gale_pic1}.
We give the real solutions to 3 significant digits
\[
  \begin{array}{c}
  (-0.800, -0.0726)\,,\ 
   ( 4.13,  5.26)\,,\ 
   ( 1.44, -0.402)\,,\ 
   ( 2.04, -0.0561)\,,\\
   ( 0.249,  0.864)\,,\rule{0pt}{13pt}
  ( 0.296,  0.602)\,,\ 
  ( 0.670,  0.250)\,,\ 
  ( 0.866,  1.54)\,,\ 
  ( 1.43,  1.42)\,,\ 
  ( 1.70,  0.634)\,.
 \end{array}
\]
These correspond (in order) to the solutions in the list~\eqref{Eq:Sols_abcde} under
the map $(a,b,c,d,e)\mapsto(be^2/4, ea^{-1}b^{-1})$.
Gale duality generalizes this equivalence.
\end{example}

Let $\Blue{\calA}=\{0,\alpha_1,\dotsc,\alpha_{l+n}\}\subset\Z^{n}$ be a collection of
$1{+}l{+}n$ integer vectors, which we consider to be exponents of (Laurent) monomials:
If $\alpha=(a_1,\dotsc,a_n)$, then $\Blue{x^\alpha}:=x_1^{a_1}\dotsb x_n^{a_n}$.
A polynomial with \Blue{{\sl support}} $\calA$ is a linear combination of monomials from 
$\calA$,
\[
    \Blue{f}\ :=\ \sum_{\alpha\in\calA} c_\alpha\, x^\alpha\,,
\]
where $c_\alpha\in\R$ -- our polynomials and functions are real.
Since we allow negative exponents, $f$ is defined on
the non-zero real numbers, $(\R^\times)^n$.
We consider systems of polynomials 
 \begin{equation}\label{Eq:Poly_system}
   f_1(x_1,\dotsc,x_n)\ =\ f_2(x_1,\dotsc,x_n)\ =\ 
    \dotsb\ =\ f_n(x_1,\dotsc,x_n)\ =\ 0\,,
 \end{equation}
in which each polynomial has the same support $\calA$.
We also assume that the system is general, by which we mean that it defines a finite set of
non-degenerate points in $(\R^\times)^n$, in that the differentials of the
polynomials $f_i$ are linearly independent at each solution.

This implies that the polynomials $f_i$ are
linearly independent over $\R$, in particular that the coefficient matrix of $n$ of the
monomials, say $x^{\alpha_1},\dotsc,x^{\alpha_n}$, is invertible.
We may then solve~\eqref{Eq:Poly_system} for these monomials to obtain a diagonal system
 \begin{equation}\label{Eq:diagonal}
   x^{\alpha_i}\ =\ \Blue{p_i}(x^{\alpha_{n+1}},\dotsc,x^{\alpha_{l+n}})\,,
   \qquad \mbox{for }i=1,\dotsc,n\,,
 \end{equation}
where the $p_i$ are degree 1 polynomials in their arguments.
This has the same form as~\eqref{Eq:Ex_poly}.

Let $\Blue{\calB}:=\{\beta_1,\dotsc,\beta_l\}\subset\Z^{l+n}$ be a basis for the free
Abelian group $\calA^\perp$ of integer linear relations among the vectors in $\calA$,
\[
   \Blue{\calA^\perp}\ :=\ \{(b_1,\dotsc,b_{l+n})\in\Z^{l+n}\mid
    b_1\alpha_1+\dotsb+ b_{l+n}\alpha_{l+n}\ =\ 0\}\,.
\]
If we write $\beta_j=(\beta_{j,i}\mid i=1,\dotsc,l{+}n)$, then
the monomials $x^{\alpha_i}$ satisfy
\[
   1\ =\ (x^{\alpha_1})^{\beta_{j,1}}\dotsb(x^{\alpha_{l+n}})^{\beta_{j,l+n}}\,,
   \qquad \mbox{for }j=1,\dotsc,l\,.
\]
Using the diagonal system~\eqref{Eq:diagonal} to substitute 
$p_i=p_i(x^{\alpha_{n+1}},\dotsc,x^{\alpha_{l+n}})$ for $x^{\alpha_i}$,
we obtain
\[
   1\ =\ p_1^{\beta_{j,1}}\dotsb p_n^{\beta_{j,n}} \cdot
    (x^{\alpha_{n+1}})^{\beta_{j,n+1}}\dotsb(x^{\alpha_{l+n}})^{\beta_{j,l+n}}\,,
   \qquad \mbox{for }j=1,\dotsc,l\,.
\]
If we now set $\Blue{y_i}:=x^{\alpha_{n+i}}$ and write
 $p_{n+i}(y)=y_i$ for $i=1,\dotsc,l$, then we get
 \begin{equation}\label{Eq:Master}
   1\ =\ p_1(y)^{\beta_{j,1}}\dotsb p_{l+n}(y)^{\beta_{j,l+n}}\,,
   \qquad \mbox{for }j=1,\dotsc,l\,.
 \end{equation}

These polynomials $p_i(y)$ define an arrangement \Blue{$\calH$} of hyperplanes in either
$\R^l$ or $\C^l$.
Because the exponents $\beta_{i,j}$ may be negative, the system~\eqref{Eq:Master} only
makes sense in the complement $M_\calH$ of this arrangement.
These rational functions, $p_1(y)^{\beta_{j,1}}\dotsb p_{l+n}(y)^{\beta_{j,l+n}}$, are
called \Blue{{\sl master functions}}, and they arise in the theory of hyperplane
arrangements. 
We give the main result of~\cite{BS_Gale}.

\begin{theorem}\label{T:GD}
 If the integer linear span $\Z\calA$ of the exponents $\calA$ is equal to $\Z^n$, then
 the association $x^{\alpha_{n+i}}\leftrightarrow y_i$ defines a bijection between
 solutions to the system of polynomials~$\eqref{Eq:Poly_system}$ in $(\C^\times)^n$
 and solutions to the system of master functions~$\eqref{Eq:Master}$ in $M_\calH$.

 This bijection restricts to a bijection between real solutions to both systems and
 further to a bijection between the positive solutions to~$\eqref{Eq:Poly_system}$ and the
 solutions to the system of master functions~$\eqref{Eq:Master}$ which lie in the positive
 chamber $\Delta:=\{y\in\R^l\mid p_i(y)>0,\ i=1,\dotsc,l+n\}$.
\end{theorem}

The bijections of Theorem~\ref{T:GD} respect the multiplicities of the solutions.
Because of this equivalence between the two systems, and because the exponents $\calA$
and $\calB$ come from Gale dual vector configurations, we call the
system~\eqref{Eq:Master} of master functions \Blue{{\sl Gale dual}} to the original polynomial
system.

Theorem~\ref{T:GD} may be strengthened as follows.
There still is a bijection between real solutions if we only require that $\Z\calA$ have
odd index in $\Z^n$ and that the span of $\{\beta_1,\dotsc,\beta_l\}$ have odd index in
$\calA^\perp$.
The bijection between positive solutions in $\R^n$ and solutions in the positive chamber
$\Delta$ in $\R^l$ remains valid if we allow $\calA$ and $\beta_i$ to be any real vectors 
such that $\calA$ spans $\R^n$ and $\{\beta_1,\dotsc,\beta_l\}$ is a basis of the real
vector space $\calA^\perp$.
The Khovanskii-Rolle continuation algorithm is naturally formulated in terms of real
exponents $\beta_i$.

%
\subsubsection{Master functions on bounded polyhedra}
Before we discuss the numerical algorithms of path-following and homotopy continuation, we
show that it is no loss to assume that this positive chamber $\Delta$ is bounded.
Suppose that $p_1(y),\dotsc,p_{l+n}(y)$ are degree 1 polynomials in $y\in\R^l$ that span
the linear space of degree 1 polynomials, let $(\beta_{i,j})\in\R^{(l+n)\times l}$ be an
array of real numbers, and consider the system of master functions
 \begin{equation}\label{Eq:MFS}
   \prod_{i=1}^{l+n} p_i(y)^{\beta_{i,j}}\ =\ 1
    \qquad\mbox{for}\quad j=1,\dotsc,l\,,
 \end{equation}
in the polyhedron
\[
   \Delta\ :=\ \{y\in\R^l\mid p_i(y)>0,\ i=1,\dotsc,l+n\}\,.
\]

Suppose that $\Delta$ is unbounded.
By our assumption on $p_1,\dotsc,p_{l+n}$, $\Delta$ is strictly convex and therefore has a
bounded face, $F$.
Then there is a degree 1 polynomial, $\Blue{q}(y):=a_0+ \alpha\cdot y$ where
$\alpha\in\R^l$ and $a_0>0$,
that is strictly positive on $\Delta$, and such that $F$ is
the set of points of $\Delta$ where $q(y)$ achieves its minimum value on $\Delta$.
Dividing by $a_0$, we may assume that the constant term of $q(y)$ is 1.

Consider the projective coordinate change $y\Rightarrow \oy$, where 
 \begin{equation}\label{Eq:PCC}
   \Blue{\oy_j}\ :=\ y_j/q(y)\qquad\mbox{for }j=1,\dotsc,l\,.
 \end{equation}
Then we invite the reader to check that 
\[
   y_j\ =\ \oy_j/r(\oy)\qquad\mbox{for }j=1,\dotsc,l\,,
\]
where $\Blue{r(\oy)}:=1-\alpha\cdot \oy$.
Note that $q(y)r(\oy)=1$.
  
The coordinate change~\eqref{Eq:PCC} manifests itself on degree 1 polynomials as follows. 
If $p(y)$ is a degree 1 polynomial, then
 \begin{equation}
   p(y)\ =\ \op(\oy)/r(\oy)
   \qquad\mbox{and}\qquad
   \op(\oy)\ =\ p(y)/q(y)\,,
 \end{equation}
where $\Blue{\op(\oy)}:=p(\oy)- p(0)\alpha\cdot \oy$.
Under the projective transformation~\eqref{Eq:PCC}, the polyhedron $\Delta$ is transformed into
$\oD$, where
\[
   \Blue{\oD}\ :=\ \{\oy\in\R^\ell\mid 
    r(\oy)>0\ \mbox{ and }\ \op_i(\oy)>0\quad\mbox{for }i=1,\dotsc,l+n\}\,.
\]

\begin{proposition}
  Under the coordinate change~\eqref{Eq:PCC}, the system of master functions~\eqref{Eq:MFS} on
  $\Delta$ is transformed into the system
\[
   \prod_{i=0}^{l+n} \op_i(\oy)^{\beta_{i,j}}\ =\ 1
    \qquad\mbox{for}\quad j=1,\dotsc,l\,,
\]
 on the polyhedron $\oD$, where $\Blue{\op_0}(\oy)=\oq(\oy)$ and 
 $\Blue{\beta_{0j}}:=-\sum_{i=1}^{l+n}\beta_{i,j}$.
 Furthermore, $\oD$ is bounded.
\end{proposition}

\begin{proof}
 If $y\in\Delta$ and $\oy\in\oD$ are related by the coordinate
 transformation~\eqref{Eq:PCC}, then $p_i(y)=\op_i(\oy)/r(\oy)$.
 Thus 
\[
  \prod_{i=1}^{l+n} p_i(y)^{\beta_{i,j}}\ =\ 
  \left(\prod_{i=1}^{l+n} \op_i(\oy)^{\beta_{i,j}}\right)\Big/r(\oy)^{\sum_i\beta_{i,j}}
  \ =\ \prod_{i=0}^{l+n} \op_i(\oy)^{\beta_{i,j}}\ ,
\]
 where $\op_0(\oy)=r(\oy)$ and $\beta_{0j}=-\sum_i \beta_{i,j}$.
 This proves the first statement.\smallskip

 All that remains is to show that $\oD$ is bounded.
 The polyhedron $\Delta$ is the Minkowski sum
 \begin{equation}\label{Eq:Rec_Cone}
   \Delta\ =\ P\ +\ R\,,
 \end{equation}
 where $P$ is a bounded polytope and $R$ is a strictly convex polyhedral cone with vertex
 the origin.
 The bounded faces of $\Delta$ are faces of $P$.
 In particular, the bounded face $F$ consisting of points where $q$ achieves its minimum
 on $\Delta$ is also a face of $P$.
 Additionally, $q$ achieves it minimum value on $R$ at the origin.
 In particular, if $0\neq v\in R$, then 
 $\alpha\cdot v=q(v)-q(0)>0$.

 Let $\pi$ be an upper bound for the ratio $\|y\|/q(y)$ for $y\in P$,
 which bounds $\|\oy\|$ when $\oy$ corresponds to a point in $P$.
 Letting $v$ range over a set of generators of the (finitely many) extreme rays of $R$
 shows that there is a positive constant $\rho$ such that 
\[
    \alpha\cdot v\ \geq\ \|v\|/\rho\qquad\mbox{for}\quad v\in R\,.
\]

 Let $\oy\in\oD$.
 Then $y:=\oy/r(\oy)$ is the corresponding point of $\Delta$.
 Writing $y=y_P+y_R$ where $y_P\in P$ and $y_R\in R$, then $q(y)=q(y_P)+\alpha\cdot y_R$
 with both terms non-negative, and we have
\[
  \|\oy\|\ =\ \left\| \frac{y}{q(y)}\right\|
  \ \leq\ \frac{\|y_P\|}{q(y)}+\frac{\|y_R\|}{q(y)}
  \ \stackrel{!}{\leq}\ 
  \frac{\|y_P\|}{q(y_P)}+\frac{\|y_R\|}{\alpha\cdot y_R}
  \ \leq\ \pi+\rho\ .
\]
 This shows that $\oD$ is bounded.
\end{proof}

We remark that this coordinate transformation $y\mapsto \overline{y}$ is a concrete and
explicit projective transformation transforming the unbounded polyhedron $\Delta$ into a
bounded polyhedron $\overline{\Delta}$.

%
\subsection{Curve tracing}\label{S:Cur_trac}
%

Let $g\colon \R^n\to\R^{n-1}$ be a smooth function such that $g^{-1}(0)$ is a smooth curve $C$.
Given a point $c\in\R^n$ on or near $C$, to move along the arc of $C$ near $c$ is
to produce a series of approximations to points along this arc.
Predictor-corrector methods~\cite{AG90}
are the standard general numerical method for this task.
They proceed as follows:
\smallskip

\noindent{\bf \underline{I Prediction}.}
  Move some specified distance (the \Blue{{\sl steplength}}) $\Delta t$ along 
  the tangent line to $C$ at $c$. 
  (The tangent direction is the kernel of the Jacobian matrix of $g$ at $C$.)
  That is, if $\widetilde{T}$ is the unit tangent vector to $C$ 
  pointing in the appropriate direction, set $\Blue{c_p}:=c+\Delta t\cdot\widetilde{T}$.\smallskip 

\noindent{\bf\underline{II Correction}.}
 Choose a linear equation $L(z)=0$ vanishing at $c_p$ and transverse to $\widetilde{T}$,
 for example $L(z):= \widetilde{T}\cdot(z-c_p)$.
 If $\Delta t$ is sufficiently small, then $c_p$ is close to a point on $C$
 where $L=0$.
 Adding $L$ to $g$ gives a system $\widetilde{g}$ with $c_p$ close to a solution, and Newton
 iterations starting at $c_p$ for the system $\widetilde{g}$ approximate
 this nearby point on $C$ with $L=0$.

 Since $C$ is smooth and $\Delta t$ small, the Jacobian matrix of $\widetilde{g}$ 
 has full rank in a cylinder around $C$ near the point $c$ and $c_p$ lies in that cylinder.
 This procedure will produce a sequence of approximations 
 $c_p,c_1, \ldots, c_k$ to a point on $C$.\smallskip

\noindent{\bf\underline{III Adaptive Steplength}.}
If correction fails to converge after a preset number (say $k=2$) of iterations, then
reduce $\Delta t$ by some factor, say $\frac{1}{2}$, and return to Step I.
Otherwise, if several consecutive predictor-corrector steps have succeeded, 
increase $\Delta t$ by some factor, say 2.  In this way,  mild curves 
are tracked quickly and wild curves are tracked cautiously.
The entire predictor-corrector procedure is terminated as a failure if 
$\Delta t$ falls below a threshold which is a function of the user's patience 
and the available numerical precision.\smallskip

\noindent{\bf\underline{IV Resetting or Stopping}.}
  If the step was a success, decide whether 
to stop, based on stopping criteria specific to the problem being solved. 
If it is not time to stop, set $c=c_k$ and go to Step I.  If the 
step was a failure, leave $c$ as is and go to Step I.
\medskip

If another branch of the curve $C$ is near the arc we are tracing and if the steplength
$\Delta t$ is large enough, then $c_p$ may fall in the basin of convergence of the
other branch.
The only general remedy for this \Blue{{\sl curve-jumping}} is to restrict both the 
maximum allowed steplength $\Delta t$ and the number of Newton corrections, typically to
2. 
Even then, curve-jumping may occur, only to be discovered with problem-specific checking 
procedures as described in Section~\ref{S:Implementation}.

There are many possible variations.
For example, the tangent predictor could be replaced (after one or more initial steps) with  
an interpolating predictor.  
For more details, see~\cite{AG90}.

%
\subsection{Homotopy continuation}\label{S:Hom_cont}
%

Predictor-corrector methods are the method of choice for homotopy 
continuation.  
They involve tracking a curve with a specific structure 
defined in a product of $\mathbb C^n$ (the space of variables, treated as $\R^{2n}$) and  
$\mathbb C$ (a parameter space).  

For the purposes of our discussion, let 
\[
   f(z)\ =\ 0\,, \text{ with }f\ \colon\ \C^n\to \C^n\,,
\]
be the \Blue{{\sl target polynomial system}} for which the isolated 
complex solutions are sought.  
Based upon the structure of the system $f$, a \Blue{{\sl start system}} 
\[
   g(z), \text{ with }g\ \colon\ \C^n\ \longrightarrow\ \C^n\,,
\]
is chosen and the \Blue{{\sl homotopy}} 
\[
   \Blue{H(z,t)}\ :=\ f(z)\cdot (1-t)+\gamma\cdot t\cdot g(z), 
   \text{ with }H\ \colon\ \mathbb C^n\times \mathbb C\to \mathbb C^n\,,
\]
is formed. 
Here, $0\neq\gamma\in\C$ is a randomly chosen complex number.
The start system $g(z)$ is chosen so that it is 
easily solved and so that its structure is related to that of $f$.
 
At $t=1$, we know the solutions of $H(z,1)=\gamma g(z)$, and we seek the 
solutions of $H(z,0)=f(z)$.
For any given value of $t$, $H(z,t)$ is a polynomial system with some number 
$N_t$ of isolated solutions.  
This number is constant, say $N$, for all except finitely many $t\in \mathbb C$.
For a general $\gamma$, restricting $H(z,t)$ to the interval $(0,1]$ gives
$N$ real arcs, one for each solution to $g(z)=0$ when $t=1$.

Some arcs will extend to $t=0$ and their endpoints give all isolated solutions of
the target system $f(z)=0$.  
Some may diverge (wasting computational time), and 
two or more arcs will converge to each singular solution.
Different choices of start system $g(z)=0$ will have different numbers of the extraneous 
divergent arcs.
This is illustrated in Figure~\ref{F:homotopy_arcs}.
\begin{figure}[htb]
\[
  \begin{picture}(121,135)(0,-25)
   \put(0,0){\includegraphics{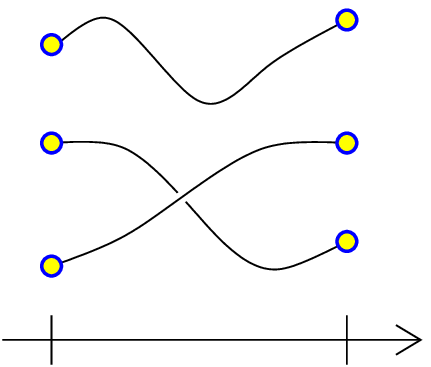}}
   \put(12,-11){$0$}   \put(97,-11){$1$}
   \put(108,14){$t$}
   \put(8,-27){No extraneous arcs}
  \end{picture}
  \qquad
  \begin{picture}(121,135)(0,-25)
   \put(0,0){\includegraphics{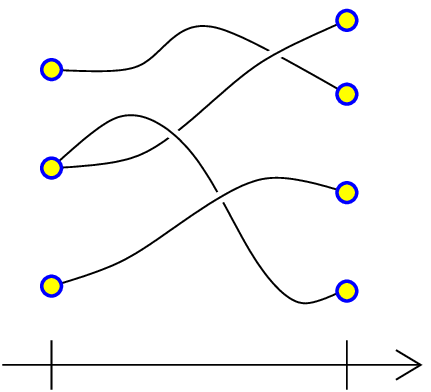}}
   \put(12,-11){$0$}   \put(97,-11){$1$}
   \put(25,-27){Singular arcs}
   \put(108,14){$t$}
  \end{picture}
  \qquad
  \begin{picture}(121,135)(0,-25)
   \put(0,0){\includegraphics{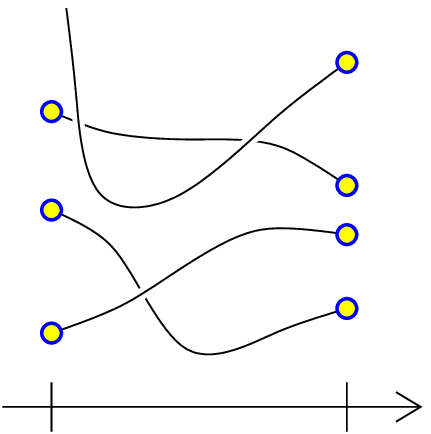}}
   \put(12,-11){$0$}   \put(97,-11){$1$}
   \put(25,-27){Divergent arc}
   \put(108,14){$t$}
  \end{picture}
\]
%
\caption{Arcs in a homotopy}
\label{F:homotopy_arcs}
\end{figure}

Here, we {\sl track paths} rather than {\sl trace curves}.
The predictor step is the same as in curve-tracing, but rather than correct 
along a line normal to the tangent, we freeze the value of $t$ and correct 
back to the curve for that value of $t$.  
For more details, see~\cite{SW05}.

%
\subsection{Curve-tracing versus path-tracking}
%

The primary difference between path-tracking in homotopy continuation and 
general curve-tracing as described in Section~\ref{S:Cur_trac} 
is that in curve-tracing there are only application-specific 
safeguards against curve-jumping.  
Homotopy continuation has a well-developed theory to avoid path-crossing.  
For almost all $\gamma$ (those that do not satisfy a certain algebraic relation), 
all paths will remain distinct until $t=0$.
As a result, all paths may be tracked to some small value of $t$ 
and compared to check for path-crossing.

A second difference involves methods to follow ill-conditioned curves.
While the curves to be tracked are generally non-singular, 
they are not necessarily well-conditioned.
These algorithms involve repeatedly solving linear systems $Ax=b$, typically 
when $A$ is some Jacobian, and the efficacy of this step depends upon the
condition number $\kappa(A)$ of $A$.
Wilkinson~\cite{Wilk} showed that we can expect 
no more than $P-\log(\kappa(A))$ digits of accuracy when solving $Ax=b$ with 
$P$ digits of precision.
For homotopy continuation, adaptive precision techniques can overcome these 
potential losses in accuracy~\cite{BHSW_AMP}.  
While this could likely be extended to general curve-tracing, the details have not yet
been worked out.

Finally, when tracing a curve with large curvature and large initial steplength 
$\Delta t$, it will take some time for $\Delta t$ to become small enough so that Newton's
method can succeed.  
There are many other situations in which curve-tracing will 
experience difficulties.  
However, the paths tracked in homotopy continuation are generally 
very well-behaved with gentle shifts in curvature, except possibly near $t=0$.
For that, \Blue{{\sl endgames}}~\cite{HV98} have been developed 
to handle singularities at $t=0$. 
Methods for handling high curvature, perhaps by adapting endgames to the 
general curve-tracing setting, are a topic for future research.

%
\section{The Khovanskii-Rolle algorithm}\label{S:Algorithm}
%
We describe the structure of the Khovanskii-Rolle algorithm in the simple case of 
finding positive solutions in Section~\ref{subsec:KR_alg_pos} and in the general 
case of finding real solutions in Section~\ref{subsec:KR_alg_real}.
Section~\ref{S:Implementation} will discuss our actual implementation of the algorithm.
Some details of various subroutines---most notably a method for following curves 
beginning at singular points in the boundary of $\Delta$---will be discussed there.

The algorithm has two elementary routines:
curve-tracing (as described in Section~\ref{S:Cur_trac}), and polynomial system solving
for a precomputation. 
We shall treat both elementary routines as function calls in our description of the
Khovanskii-Rolle continuation algorithm.
The algorithm proceeds iteratively with the starting points for curve-tracing being
either the output of the previous stage of the algorithm or solutions to certain
polynomial systems restricted to faces of the domain polytope, $\Delta$.

We first describe the main continuation routine.
This is a curve-tracing algorithm that uses solutions to one system of equations 
as input to find the solutions to another.
We then describe our algorithm for finding all solutions to a system of master functions 
within the positive chamber, followed by the general method for all real solutions.

%
\subsection{Continuation step}
%

Suppose that we have smooth functions $f_1,\dotsc,f_{l-1},g$ on a domain $\Delta\subset\R^l$
with finitely many common zeroes in $\Delta$.  Also suppose that the zero locus of 
$f_1,\dotsc,f_{l-1}$ is a smooth curve \Blue{$C$} in $\Delta$.
Let \Blue{$J$} be the Jacobian determinant $\Jac(f_1,\dotsc,f_{l-1},g)$ of the functions
$f_1,\dotsc,f_{l-1},g$. 

The goal of this step is to trace $C$ to find all
common zeroes of $f_1,\dotsc,f_{l-1},g$ in $\Delta$, taking as input the zeroes of $J$ on
the curve $C$ and the points where $C$ meets the boundary of $\Delta$.
This uses an extension of Rolle's theorem due to Khovanskii.
We only use a very simple version of this theorem, which may be deduced from the classical Rolle
theorem.

\begin{proposition}[Khovanskii-Rolle Theorem]
 Between any two zeroes of $g$ along an arc of $C$ there is at least one
 zero of $J$.
\end{proposition}

Figure~\ref{F:KhRoEx} illustrates the Khovanskii-Rolle Theorem when $l=2$.
Note that the Jacobian determinant $J$ here is the exterior product $df\wedge dg$.
\begin{figure}[htb]
\[
  \begin{picture}(270,130)
   \put(0,0){\includegraphics{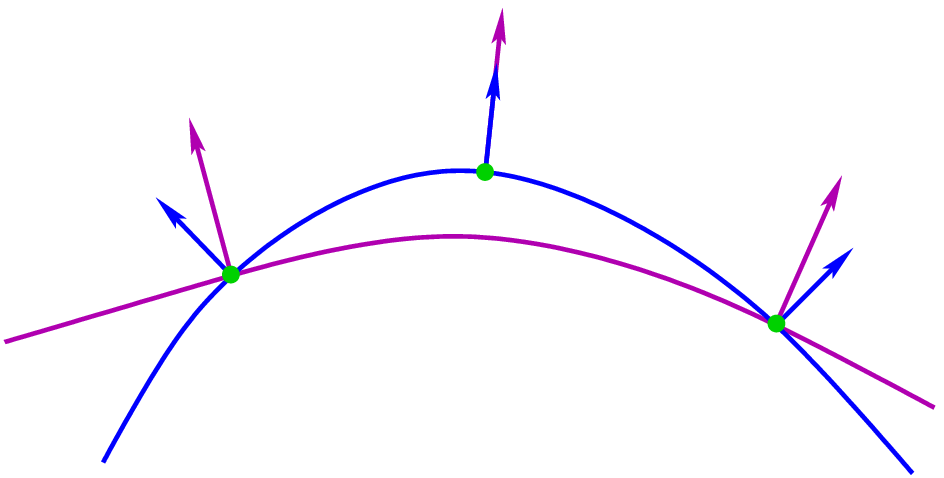}}
   \put(37,6){$\Blue{C}\colon \Blue{f}=0$} \put(-5,29){$\Magenta{g}=0$}
   \put(65,48){\ForestGreen{$a$}}
   \put(31,69){\Blue{$df$}}  \put(60,97){\Magenta{$dg$}}
   \put(136,77){\ForestGreen{$c$}}
   \put(122,111){\Blue{$df$}}  \put(149,124){\Magenta{$dg$}}
   \put(217,33){\ForestGreen{$b$}}
  \put(245,55){\Blue{$df$}}  \put(222,79){\Magenta{$dg$}}
  \end{picture}
\]
\caption{$\Blue{df}\wedge\Magenta{dg}(\ForestGreen{a})<0$, \,
         $\Blue{df}\wedge\Magenta{dg}(\ForestGreen{c})=0$, \,  and 
         $\Blue{df}\wedge\Magenta{dg}(\ForestGreen{b})>0$.}
\label{F:KhRoEx}
\end{figure}

Khovanskii-Rolle continuation begins with a solution of $J=0$ on $C$ and traces $C$
in both directions until either it encounters a solution of $g=0$ on $C$ or some other 
stopping criterion is satisfied, as described below.

This clearly yields all solutions to $g=0$ on compact arcs of $C$.  For 
non-compact arcs that have two or more solutions, the Khovanskii-Rolle Theorem
guarantees that solutions of $J=0$ will be adjacent to solutions of $g=0$.
To ensure that all zeroes of $g=0$ on $C$ are found, the algorithm also traces
non-compact components of $C$ starting from the points where $C$ meets the boundary of
$\Delta$. 
Here is a more complete pseudocode description.

\begin{alg}[Continuation Algorithm]\label{Alg:Continuation_step}{\rm  \ 

 Suppose that $f_1,\dotsc,f_{l-1}$ are differentiable functions defined on a domain
 $\Delta\subset\R^l$ 
 that define a smooth curve $C$, and $g$ is another differentiable function with finitely many 
 simple zeroes on $C$.
 Let $J:=\Jac(f_1,\dotsc,f_{l-1},g)$ be the Jacobian determinant of these functions.\smallskip

 \noindent{\sc Input:}  Solutions $S$ of 
    $f_1=\dotsb=f_{l-1}=J=0$ in the interior of $\Delta$ and the set $T$ of points 
    where $C$ meets the boundary of $\Delta$.

 \noindent{\sc Output:}  All solutions $U$ to $f_1=\dotsb=f_{l-1}=g=0$ in
    the interior of $\Delta$. 

 For each $s$ in $S$, follow $C$ in both directions from $s$ until one of the following  
   occurs.
 \begin{equation}\label{Eq:stopping}
  \begin{tabular}{cl}
    (1) & A solution $u\in U$ to $g=0$ is found,\\
    (2) & A solution $s'\in S$ to $J=0$ is found, or\\
    (3) & A point $t\in T$ where $C$ meets the boundary of $\Delta$ is found.
  \end{tabular}\hspace{20pt}
 \end{equation}

 For each $t$ in $T$, follow 
$C$ in the direction of the interior of $\Delta$ until one of the stopping 
criteria~$\eqref{Eq:stopping}$ occurs.
Each point $u\in U$ is found twice.
}
\end{alg}

The computation of the sets $S$ and $T$ of solutions is presented 
after the full Khovanskii-Rolle algorithm.

\begin{proof}[Proof of correctness]
  By the Khovanskii-Rolle
  Theorem, if we remove all points where $J=0$ from $C$ to obtain a curve $C^\circ$,
  then no two solutions of $g=0$ will lie in the same connected component of $C^\circ$.
  No solution can lie in a compact component (oval) of $C^\circ$, as this is also 
  an oval of $C$, and the number of solutions $g=0$ on an oval is
  necessarily even, so the Khovanskii-Rolle Theorem would imply that $J=0$ has a solution on
  this component. 

  Thus the solutions of $g=0$ lie in different connected components of $C^\circ$, and 
  the boundary of each such component consists of solutions to $J=0$ on $C$ and/or points where  
  $C$ meets the boundary of $\Delta$.
  This shows that the algorithm finds all solutions $g=0$ on $C$, and that it will find each twice,
  once for each end of the component of $C^\circ$ on which it lies.
\end{proof}

Each unbounded component of $C$ has two ends that meet the boundary of $\Delta$.
Writing $\calV(f_1,\dotsc,f_{l-1},g)$ for the common zeroes in $\Delta$ to
$f_1,\dotsc,f_{l-1},g$, and $\mbox{ubc}(C)$ for the unbounded components of $C$, 
we obtain the inequality
 \begin{equation}\label{Eq:path_bound}
  2\#\calV(f_1,\dotsc,f_{l-1},g)\ \leq\ \#\mbox{paths followed}\ =\ 
    2\#\mbox{ubc}(C)
    \ +\ 2\#\calV(f_1,\dotsc,f_{l-1},J)\,,
 \end{equation}
as each solution is obtained twice.
This gives a bound on the number of common zeroes of $f_1,\dotsc,f_{l-1},g$ in $\Delta$
and it will be used to estimate the complexity of this algorithm.

%
\subsection{Khovanskii-Rolle algorithm for positive solutions}\label{subsec:KR_alg_pos}
The Khovanskii-Rolle algorithm yields all solutions to a system of equations
of the form
 \begin{equation}\label{Eq:SMF}
   \Blue{\psi_j(y)}\ :=\ \prod_{i=0}^{l+n} p_i(y)^{\beta_{i,j}}\ =\ 1,\qquad\mbox{ for } j=1,\dotsc,l\,,
 \end{equation}
in the polyhedron 
 \begin{equation}\label{Eq:Delta}
   \Delta\ :=\ \{y\in\R^l \mid p_i(y)>0\}\,,
 \end{equation}
which we assume is bounded.
Here $p_0(y),\dotsc,p_{l+n}(y)$ are degree 1 polynomials that are general in that no
$l+1$ of them have a common solution.
Also, $\beta_{i,j}$ are real numbers that are sufficiently general in a manner
that is explained below. 

We first make some observations and definitions.
In $\Delta$, we may take logarithms of the master functions in~\eqref{Eq:SMF} to obtain
the equivalent system
\[
   \Blue{\varphi_j(y)}\ :=\ \log \psi_j(y)\ =\ 
   \sum_{i=0}^{l+n} \beta_{i,j} \log p_i(y)\ \ =\ 0\,
   \qquad\mbox{for }j=1,\dotsc,l\,.
\]
The Khovanskii-Rolle algorithm will find solutions to these equations by iteratively applying 
Algorithm~\ref{Alg:Continuation_step}.
For $j=l,l{-}1,\dotsc,2,1$ define the Jacobian determinants
\[
   \Blue{\widetilde{J}_j}\ :=\ \Jac(\varphi_1,\dotsc,\varphi_{j};\,
                         \widetilde{J}_{j+1},\dotsc,\widetilde{J}_l)\,.
\]
Also, write the rational function $\psi_j(y)=\psi_j^+(y)/\psi_j^-(y)$ as a quotient of two
polynomials $\psi_j^+$ and $\psi_j^-$.
Writing $f_j(y)$ for the difference $\psi_j^+(y)-\psi_j^-(y)$,
$\psi_i(y)=1$ is equivalent to $f_j(y)=0$.

\begin{alg}[Khovanskii-Rolle Continuation]\label{Alg:Kh_Ro}{\rm  \ 

 Suppose that $\varphi_j$ and $\widetilde{J}_j$ for $j=1,\dotsc,l$ are as above.
 For each $j=0,\dotsc,l$, let $S_j$ be the solutions in $\Delta$ to
\[
    \varphi_1\ =\ \dotsb\ =\ \varphi_{j}\ =\ 
      \widetilde{J}_{j+1}\ =\ \dotsb\ =\ \widetilde{J}_l\ =\ 0\,,
\]
 and for each $j=1,\dotsc,l$, let \Blue{$T_j$} be the set of solutions to 
 \begin{equation}\label{Eq:boundary}
    f_1\,=\,\dotsb\,=\,f_{j-1}\ =\ 0\qquad\mbox{and}\qquad
    \widetilde{J}_{j+1}\,=\,\dotsb\,=\,\widetilde{J}_l\ =\ 0
 \end{equation}
  in the boundary of $\Delta$.\smallskip

 \noindent{\sc Input:} 
   Solutions $S_0$ of $\widetilde{J}_1=\dotsb=\widetilde{J}_l=0$ in $\Delta$, and 
   sets $T_1,\dotsc,T_l$.

 \noindent{\sc Output:}  Solution sets $S_1,\dotsc,S_l$.\smallskip 

 For each $j=1,\dotsc,l$, apply the Continuation
 Algorithm~\ref{Alg:Continuation_step} with 
\[
   (f_1,\dotsc,f_{l-1})\ =\ 
   (\varphi_1, \dotsc,\varphi_{j-1}\,,\,
    \widetilde{J}_{j+1}, \dotsc,\widetilde{J}_l)
\]
  and $g=\varphi_j$.
  The inputs for this are the sets $S_{j-1}$ and $T_j$, and the output is the set $S_j$.
}
\end{alg}

The last set computed, $S_l$, is the set of solutions to the system of master
functions~\eqref{Eq:SMF} in $\Delta$.
We describe our implementation for positive solutions when $l=2$ in 
Section~\ref{S:Implementation} and illustrate the algorithm in 
Section~\ref{S:examples}.

\begin{proof}[Proof of correctness]
 Note that for each $j=1,\dotsc,l$, the $j$th step finds the solution set $S_j$, and
 therefore preforms as claimed by the correctness of the Continuation
 Algorithm~\ref{Alg:Continuation_step}. 
 The correctness for the Khovanskii-Rolle algorithm follows by induction on $j$.
\end{proof}

The Khovanskii-Rolle algorithm requires the precomputation of the sets $S_0$ and
$T_1,\dotsc,T_l$.
The feasibility of this task follows from Lemma~3.4 in~\cite{BS},
which we state below. 
By our assumption on the generality of the polynomials $p_i$, a face of $\Delta$ of
codimension $k$ is the intersection of $\Delta$ with $k$ of the hyperplanes $p_i(y)=0$.

\begin{proposition}\label{P:estimation}
  Let $p_i,i=0,\dotsc,l{+}n$, $\varphi_j,\widetilde{J}_j,j=1,\dotsc,l$, and $\Delta$ be as 
  above. 
\begin{enumerate}
 \item The solutions $T_j$ to~\eqref{Eq:boundary}
  in the boundary of $\Delta$ are the solutions to 
 \[
    \widetilde{J}_{j+1}=\dotsb=\widetilde{J}_l=0
 \]
 in the codimension-$j$ faces of $\Delta$.

\item \rule{0pt}{24pt}${\displaystyle \widetilde{J}_j\cdot 
        \prod_{i=0}^{l+n}p_i(y)^{2^{l-j}}}$ is a polynomial
      of degree $ 2^{l-j}n$.
\end{enumerate}
\end{proposition}

By~Proposition~\ref{P:estimation}(1), solving the system~\eqref{Eq:boundary} in the boundary of
$\Delta$ is equivalent to solving (at most) $\binom{l+n+1}{j}$ systems of the form
 \begin{equation}\label{Eq:simple_system}
   p_{i_1}\ =\ \dotsb\ =\ p_{i_j}\ \ =\ \ 
   \widetilde{J}_{j+1}\ =\ \dotsb\ =\ \widetilde{J}_l\ =\ 0\,
 \end{equation}
and then discarding the solutions $y$ for which $p_i(y)<0$ for some $i$.
Replacing each Jacobian $\widetilde{J}_j$ by the polynomial
$\Blue{J_j}:=\widetilde{J}_j\cdot\prod_{i=0}^{l+n}p_i(y)^{2^{l-j}}$,
and using $p_{i_1}(y) = \dotsb = p_{i_j}=0$ to eliminate $j$ variables,
we see that~\eqref{Eq:simple_system} is a polynomial system with B\'ezout number
\[
   n \cdot 2n\cdot 4n \dotsb\ 2^{l-j-1}n\ =\ 2^{\binom{l-j}{2}}n^{l-j}
\]
in $l-j$ variables.
Thus the number of solutions $T_j$ to the system~\eqref{Eq:boundary} in the boundary of
$\Delta$ is at most $2^{\binom{l-j}{2}}n^{l-j}\binom{l+n+1}{j}$.

Likewise $S_0$ consists of solutions in $\Delta$ to the system~\eqref{Eq:simple_system} when
$j=0$ and so has at most the B\'ezout  number $2^{\binom{l}{2}}n^l$ solutions.
Thus the inputs to the Khovanskii-Rolle Algorithm are solutions to polynomial systems in $l$ or
fewer variables.

\begin{theorem}\label{Th:Kh-Ro_complexity}
 The Khovanskii-Rolle Continuation Algorithm finds all solutions to the system~\eqref{Eq:SMF}
 in the bounded polyhedron $\Delta$~\eqref{Eq:Delta}, when the degree $1$ polynomials $p_i(y)$
 and exponents $\beta_{i,j}$ are general as described.
 It accomplishes this by solving auxiliary polynomial systems~$\eqref{Eq:simple_system}$ and
 following implicit curves.   For each $j=0,\dotsc,l{-}1$, it will solve at most 
 $\binom{l{+}n{+}1}{j}$ polynomial systems in
 $l{-}j$ variables, each having B\'ezout number $2^{\binom{l-j}{2}}n^{l-j}$.
 In all, it will trace at most
 \begin{equation}\label{Eq:Num_curves}
    l\,2^{\binom{l}{2}+1}n^l\ +\ \sum_{j=1}^l (l{+}1{-}j) 2^{l-j}n^{l-j}\tbinom{l{+}n{+}1}{j}
    \ \ <\ \ l\, \frac{e^2+3}{2}2^{\binom{l}{2}}n^l
 \end{equation}
 implicit curves in $\Delta$.
\end{theorem}

\begin{proof}
 The first statement is a restatement of the correctness of the Khovanskii-Rolle Continuation
 Algorithm.
 For the second statement, we enumerate the number of paths, following the
 discussion after Proposition~\ref{P:estimation}.
 Let $\Blue{s_j}, \Blue{t_j}$ be the number of points in $S_j$ and $T_j$, respectively, and
 let $r_j$ be the number of paths followed in the $j$th step of the Khovanskii-Rolle
 Continuation Algorithm.

 By~\eqref{Eq:path_bound}, $r_j=t_j+2s_{j-1}$ and $s_j\leq \frac{1}{2}t_j+s_{j-1}$.
 So $r_j \leq t_j+\dotsb+t_1 + 2s_0$, and 
\[
  r_1+\dotsb+r_l\ \leq\ 2s_0 + \sum_{j=1}^l (l{+}1{-}j) t_j\,.
\]
 Using the estimates
 \begin{equation}\label{Eq:est}
   s_0\ \leq\ 2^{\binom{l}{2}}n^l
   \qquad\mbox{and}\qquad
   t_j\ \leq\ 2^{\binom{l-j}{2}}n^{l-j}\tbinom{l+n+1}{j}\,,
 \end{equation}
 we obtain the estimate on the left of~\eqref{Eq:Num_curves}.
 Since $l{+}1{-}j\leq l$, we bound it by 
\[
  2l\Bigl(  2^{\binom{l}{2}}n^l\ +\ \frac{1}{2}
      \sum_{j=1}^l  2^{l-j}n^{l-j}\tbinom{l{+}n{+}1}{j}\Bigr)\,,
\]
 which is bounded by  $l\, \frac{e^2+3}{2}2^{\binom{l}{2}}n^l$, by Lemma 3.5 in~\cite{BS}.
\end{proof}

\begin{remark}\label{Rem:improvement}
 The bound~\eqref{Eq:Num_curves} on the number of paths to be followed is not sharp.
 First, not every system of the linear polynomials
\[
   p_{i_1}(y)\ =\ p_{i_2}(y)\ =\ \dotsb\ =\ p_{i_j}(y)\ =\ 0\,,
\]
 defines a face of $\Delta$.
 Even when this defines a face $F$ of $\Delta$, only the solutions to~\eqref{Eq:simple_system}
 that lie in $F$ contribute to $T_j$, and hence to the number
 of paths followed.
 Thus any slack in the estimates~\eqref{Eq:est} reduces the number of paths to
 be followed.
 Since these estimates lead to the fewnomial bound for $s_l$, we see that the Khovanskii-Rolle
 Continuation Algorithm naturally takes advantage of any lack of sharpness in the fewnomial 
 bound.

 This may further be improved if $\Delta$ has $m<l+n+1$ facets for then
 the binomial coefficients in~\eqref{Eq:Num_curves} become $\binom{m}{2}$.
\end{remark}

%
\subsection{Khovanskii-Rolle algorithm for all real solutions}\label{subsec:KR_alg_real}
%

Section~\ref{subsec:KR_alg_pos} describes how to find solutions to a system
of master functions~\eqref{Eq:SMF} in the polyhedron $\Delta$~\eqref{Eq:Delta}, which is
assumed bounded.
Through Gale duality and the coordinate transformation~\eqref{Eq:PCC}, this gives a method 
to find all positive real solutions to a system of polynomial equations~\eqref{Eq:Poly_system}.

To find all real solutions to a system of polynomial equations or of master functions, one
could simply repeat this process for every chamber in the complement of the hyperplanes
$p_i(y)=0$ for $i=1,\dotsc,l{+}n$.
This is however inefficient as our method (homotopy continuation) for computing the
sets $S_0$ and $T_i$ for $i=1,\dotsc,l$ of starting points for one chamber gives the starting
points for all chambers.
Besides careful bookkeeping, this requires some 
projective coordinate transformations so that the tracking occurs in bounded chambers.

Each point in $T_j$ is incident upon $2^j$ chambers, and thus is the starting point for $2^j$
arcs to be followed in Algorithm~\ref{Alg:Kh_Ro}. 
Surprisingly, this has a mild effect on the complexity, requiring only that
we replace the $e^2$ in~\eqref{Eq:Num_curves} by $e^4$.
This observation, which was made while developing this Khovanskii-Rolle algorithm,
was the genesis of the bound in~\cite{BBS}.

%
\section{Examples}\label{S:examples}
%

We first illustrate the Khovanskii-Rolle algorithm and our implementation on the master function
system of Example~\ref{Ex:Gale}. 
Write the system~\eqref{Eq:first_master} of master functions in logarithmic form,
$\varphi_1(y)=\varphi_2(y)=0$, for $y\in\Delta$, which is the heptagon depicted in
Figure~\ref{F:Gale_pic1}. 
Here, we have $l=2$ and $n=5$ with $7$ linear polynomials.
 \begin{eqnarray*}
  \varphi_1(y)&=&
      -\log(y)+\log(2x{+}2y{-}1)-2\log(4x)+\log\Bigl(\tfrac{1+x-y}{2}\Bigr)\\
   && \qquad\qquad\qquad-2\log\Bigl(\tfrac{6-x-3y}{4}\Bigr)+2\log\Bigl(\tfrac{8-3x-2y}{2}\Bigr)
   -\log(3{-}2x{+}y)\,.\vspace{2pt}\\
  \varphi_2(y)&=&
      \log(y)+6\log(2x{+}2y{-}1)-3\log(4x)+6\log\Bigl(\tfrac{1 +  x - y}{2}\Bigr)\\
   && \qquad\qquad\qquad  -2\log\Bigl(\tfrac{6-x-3y}{4}\Bigr)+7\log\Bigl(\tfrac{8-3x-2y}{2}\Bigr)
   +\log(3{-}2x{+}y)\,.
\end{eqnarray*}
The polynomial forms $J_2,J_1$ of the Jacobians 
are (omitting the middle 60 terms from $J_1$), 
 \begin{eqnarray*}
  J_2&=& -168x^5-1376x^4y+480x^3y^2-536x^2y^3-1096xy^4+456y^5+1666x^4+2826x^3y\\
   &&+3098x^2y^2+6904xy^3-1638y^4-3485x^3-3721x^2y-15318xy^2-1836y^3\\
   && +1854x^2+8442xy+9486y^2-192x-6540y+720\,.\\
   J_1&=&10080x^{10}-168192x^9y-611328x^8y^2-\quad\ \dotsb\quad\ +27648x+2825280y\,.
%
%
\end{eqnarray*}

\begin{remark}
 Instead of the polynomial form $J_2$ of the Jacobian of $\varphi_1$ and
 $\varphi_2$, we could use the Jacobian $J(f,g)$ of $f$ and $g$~\eqref{Eq:Mfg}.
 This however has degree 25 with 347 terms and coefficients of the order of $10^{17}$.
 The Jacobian of this and $f$ has degree 29 and 459 terms.
 This control on the degree of the Jacobians is the reason that we use the logarithms of the
 master function in the formulation of the Khovanskii-Rolle Algorithm.
\end{remark}

 We now describe the Khovanskii-Rolle algorithm on this example.

\noindent{\bf Precomputation.}
 We first find all solutions $S_0$ to $J_1=J_2=0$ in the heptagon $\Delta$,
 and all solutions $J_2=0$ in the boundary of the heptagon.
 Below are the curves $J_2=0$ and $J_1=0$ and the heptagon.
 The curve $J_2=0$ consists of the four arcs indicated.
 The remaining curves in this picture belong to $J_1=0$.
 On the right is an expanded view in a neighborhood of the lower right vertex,
 $(\frac{3}{2},0)$.
\[
  \begin{picture}(390,170)(-30,-10)
   \put(0,0){\includegraphics[height=160pt]{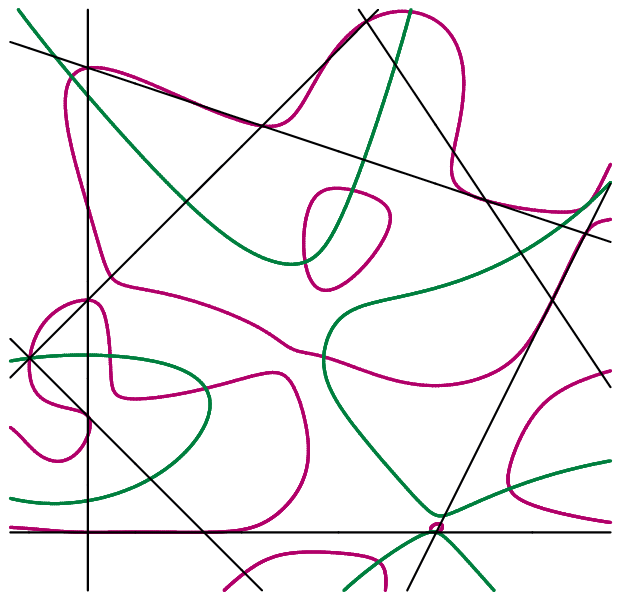}}
   \put(215,0){\includegraphics[height=160pt]{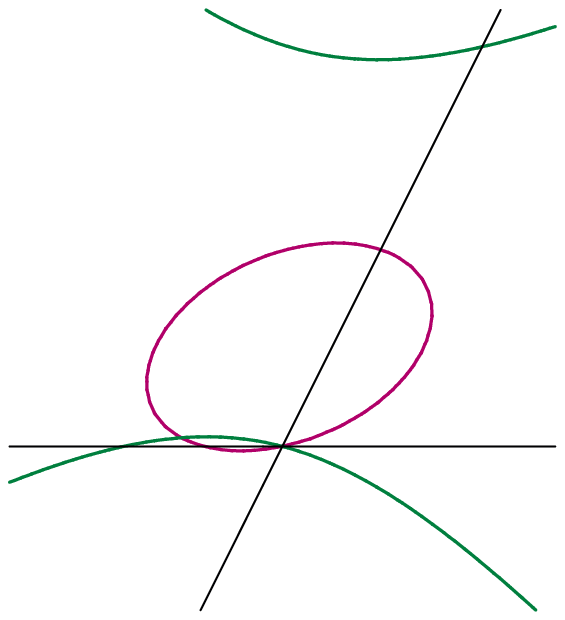}}

   \put(-26,152){$J_2$}\put(-13,157){\vector(1,0){15}}
   \put(-30, 20){$J_2$}\put(-17,25){\vector(1,0){15}}

   \put(152,-2){$J_2$}\put(150,3){\vector(-1,0){15}}\put(150,8){\vector(-1,1){16}}


   \put(263,97){$J_1$}
   \put(300,135){$J_2$}
   \put(220,25){$J_2$}
   \put(283,25){$(\frac{3}{2},0)$}

  \end{picture}
\]
A numerical computation finds 50 common solutions to $J_1=J_2=0$ with 26 real.
Only six solutions lie in the interior of the heptagon with one on the boundary at the
vertex $(3/2,0)$.
There are 31 points where the curve $J_2=0$ meets the lines supporting the boundary of the
heptagon, but only eight lie in the boundary of the hexagon.
This may be seen in the pictures above.

\noindent{\bf First continuation step.}
Beginning at each of the six points in $\Delta$ where $J_1=J_2=0$, the algorithm traces
the curve in both directions, looking for a solution to $\varphi_1=J_2=0$.
Beginning at each of the eight points where the curve $J_2=0$ meets the boundary, it
follows the curve into the interior of the heptagon, looking for a solution to
$\varphi_1=J_2=0$. 
In tracing each arc, it either finds a solution, a boundary point, or another
point where $J_1=J_2=0$. 
We may see that in the picture below, which shows the curves $\varphi_1=0$ and $J_2=0$,
as well as the points on the curve $J_2=0$ where $J_1$ also vanishes.
\[
  \begin{picture}(200,170)(-30,-10)
   \put(0,0){\includegraphics[height=160pt]{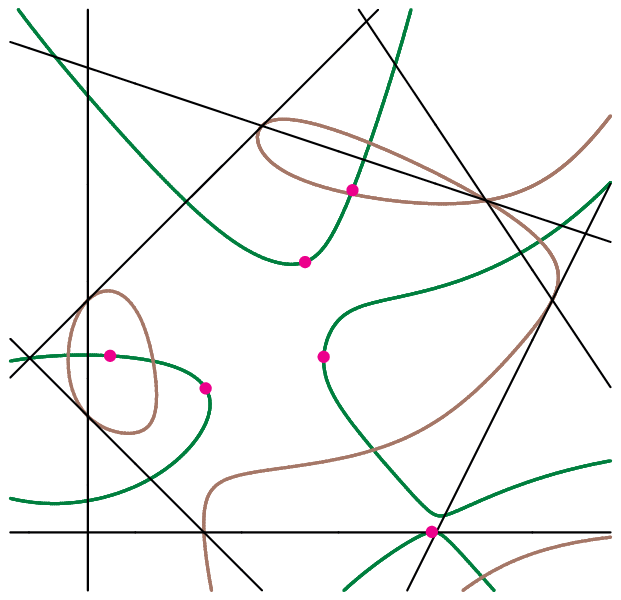}}
   \put(-26,152){$J_2$}\put(-13,157){\vector(1,0){15}}
   \put(-30, 20){$J_2$}\put(-17,25){\vector(1,0){15}}

   \put(154, -4){$J_2$}\put(152,1){\vector(-1,0){15}}\put(152,6){\vector(-1,1){18}}

    \put(-25,67){$\varphi_1$} \put(-8,68){\vector(1,0){24}}
    \put(42, 3){$\varphi_1$} 
     \put(150, 135){$\varphi_1$} 
     \put(192,13){$\varphi_1$} \put(188,15){\vector(-1,0){20}}
 
  \end{picture}
\]
In this step, $2\cdot 6 + 8=20$ arcs are traced.
The three solutions of $J_2=\varphi_1=0$ will each be found twice, and 14 of the tracings
will terminate with a boundary point or a point where $J_1=0$.

\noindent{\bf Second Continuation step.}
This step begins at each of the three points where $\varphi_1=J_2=0$ that were found
in the last step, as well as at each of the six points where $\varphi_1=0$ meets the
boundary of the heptagon (necessarily in some vertices).
\[
  \begin{picture}(230,160)(-27,0)
     \put(0,0){\includegraphics[height=160pt]{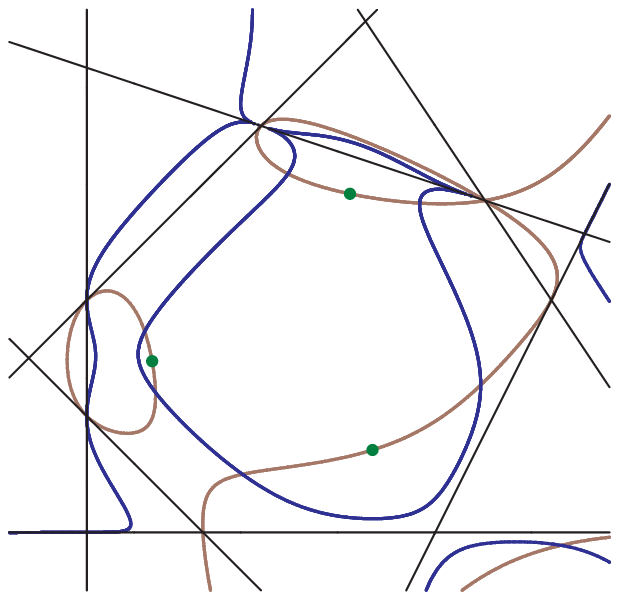}}
    \put(-25,65){$\varphi_1$} \put(-8,66){\vector(1,0){22}}
    \put(42, 3){$\varphi_1$} 
     \put(150, 135){$\varphi_1$} 
     \put(192,13){$\varphi_1$} \put(188,15){\vector(-1,0){20}}
   \end{picture}
\]

Curve tracing, as described in more detail in the next section, can be carried out even 
in the presence of singularities, as in the case of the curves initiating at 
vertices.  In this case, this final round of curve-tracing revealed all six 
solutions within the heptagon, as anticipated.  Furthermore, each solution 
was discovered twice, again, as anticipated. 

By Theorem~\ref{Th:Kh-Ro_complexity}, the bound on the number of paths followed (using
$7=l+n$ in place of $l+n+1$ in the binomials as in Remark~\ref{Rem:improvement}) is
\[
   2\cdot 2\cdot 2^{\binom{2}{2}}\cdot 5^2\ +\ 
   2\cdot 2^{\binom{1}{2}}\cdot 5^1\cdot\tbinom{7}{1}\ +\ 
   1\cdot 2^{\binom{0}{2}}\cdot 5^0\cdot\tbinom{8}{2}\ =\ 298\,.
\]
By Theorem 3.10 in~\cite{BS}, the fewnomial bound in this case is 
\[
  2\cdot 5^2+ \lfloor\tfrac{(5+1)(5+3)}{2}\rfloor\ =\ 
  74\,.
\]
In contrast, we only traced $20+12=32$ curves to find the six solutions in the heptagon.
The reason for this discrepancy is that this bound is pessimistic and the
Khovanskii-Rolle Continuation Algorithm exploits any slack in it.

\noindent{\bf Timings and comparison to existing software.}\label{SS:timings}
The system of Laurent polynomials~\eqref{Eq:Ex_poly} was converted into a system of 
polynomials by clearing denominators.  It was then run through PHCpack, Bertini, and the 
proof-of-concept implementation 
described in the next section.  As described in Section~\ref{S:Background}, this system has 
102 regular solutions, 10 of which are real.  All runs of this section were performed on 
a 2.83~GHz running CentOS with Maple 13, Bertini 1.1.1, and PHCpack v2.3.48.

In blackbox mode, PHCpack used polyhedral methods and found 102 regular solutions in 
around 2.5 
seconds, though it only classified eight of them as real.  Using all default settings 
and a 5-homogeneous start system, 
Bertini found all 102 regular solutions and identified the 10 that are real.  However, 
because of the use of adaptive precision, this took around 23 seconds.  Using 
fixed low precision, this time dropped to around 9 seconds while still identifying 
the solutions correctly.  

The implementation described in the next section found all 
positive solutions of the Gale dual system of master functions~\ref{Eq:Mfg} in around 
15 seconds using safe settings.  Almost all of this time was spent in computing the sets 
$S_0$ and $T_1$ in Bertini, using adaptive precision for security.  By changing to 
fixed low precision for the Bertini portions of the computation, the timing for the 
entire run fell to around 1.4 seconds -- the shortest time of all runs described here.
Though such efficiency is welcome, security is more valuable.  It is expected that more 
sophisticated software than that described in the next section will be more 
efficient.

\noindent{\bf A more extreme example.}\label{SS:timings2}
A polynomial system with high degree, many complex solutions, and few real solutions 
further illustrates the value of Khovanskii-Rolle continuation.  For example, 
consider the system of Laurent polynomials

\begin{equation}\label{Eq:hex7663_poly}
 \begin{array}{rclcrcl}
10500       - tu^{492} - 3500t^{-1}u^{463}v^5w^5 &=& 0\\
10500 - t               - 3500t^{-1}u^{691}v^5w^5 &=& 0\\
14000 - 2t + tu^{492} - 3500v &=& 0\\
14000 + 2t - tu^{492} - 3500w &=& 0.\\
 \end{array}  
\end{equation}

By solving a set of master functions Gale dual to (\ref{Eq:hex7663_poly})
for 0, we obtain 

 \begin{equation}\label{Eq:hex7663_master}
  \begin{array}{rcl}
    3500^{12}x^8y^4(3-y)^{45} - (3-x)^{33}(4-2x+y)^{60}(2x-y+1)^{60} &=& 0\\
    3500^{12}x^{27}(3-x)^8(3-y)^4 - y^{15}(4-2x+y)^{60}(2x-y+1)^{60} &=& 0.\\
  \end{array}
 \end{equation}

System (\ref{Eq:hex7663_poly}) has 7663 complex solutions but only six positive 
real solutions.
PHC computes these solutions in 39 minutes, 38 seconds while the implementation 
of the next section takes only 23 seconds to find them, using safe settings.

%
\section{Implementation}\label{S:Implementation}
%

We have implemented the Khovanskii-Rolle Algorithm to find all solutions to a system of master
functions in a bounded polyhedron $\Delta$, but only when $l=2$.
This proof-of-concept implementation relies on external calls to 
the Bertini software package~\cite{BHSW06} as a polynomial system solver.  
The implementation is in a Maple script that is publicly available at~\cite{KRWeb}. 
We plan to implement a general version of the Khovanskii-Rolle algorithm 
in the future.

We describe some aspects of this implementation, including the precomputation to find the
solution sets $S_0$ and $T_1$ and the curve-tracing method from these points and from $S_1$, 
all of which are smooth points on the traced curves.
We also discuss tracing curves from the vertices $T_2$, which is non-trivial as 
these curves are typically singular at the vertices.
Lastly, we discuss procedures for checking the output.

%
\subsection{Polynomial system solving}

The precomputation of $S_0$ and $T_1$ for Algorithm~\ref{Alg:Kh_Ro} involves finding the 
solutions of systems of Jacobians~\eqref{Eq:simple_system} within $\Delta$ and on its
boundary. 
For this, we use the system Bertini~\cite{BHSW06}.
For each system to be solved, the maple script
creates a Bertini input file, calls Bertini, and collects
the real solutions from an output file.  
Bertini, as with all homotopy methods, finds all complex solutions.
However, as explained previously, this overhead may be much less than would be 
encountered in the direct use of homotopy methods to solve the original 
system.

%
\subsection{Curve-tracing from smooth points}

The points of $S_0$ and $T_1$ from which we trace curves in the first step of the
algorithm, as well as the points $S_2$ used in the second step, are smooth points of the
curves $J_2=0$ and $\varphi_1=0$, respectively.
This curve-tracing proceeds as in Section~\ref{S:Cur_trac}.  
Indeed, suppose we are tracing a curve $C$ in the polytope $\Delta$, starting from points where $C$
meets the boundary of $\Delta$ and from interior starting points where a Jacobian determinant $J$
 vanishes, and we seek points of $C$ where some function $g$ vanishes.
Then, as described in the  Continuation Algorithm~\ref{Alg:Continuation_step}, there are 
three basic stopping criteria:
\begin{enumerate}

\item The tracer passes a point where $g=0$, which is a point we are looking for.

\item The tracer passes a point where $J=0$, which is another starting point.

\item The tracer leaves the polytope.
\end{enumerate}

The computation of the tangent and normal lines is straightforward, as is the 
linear algebra required for curve-tracing.  
Each predictor-corrector step gives a point $p$ near $C$.  
We compute the values of $g, J$, and the degree 1 polynomials $p_i$
defining $\Delta$ at $p$.
A sign change in any indicates a stopping criteria has been met.
When this occurs, bisection in the steplength $\Delta t$ 
from the previous approximation is used to refine the point of interest.

It may seem that these (sign-based) criteria will fail when there are clustered
(or multiple) solutions to $g=0$ on the curve between the current and previous
approximation.
However, the Khovanskii-Rolle Theorem implies there will be zeroes of the Jacobian determinant
$J$ interspersed between these solutions (or coinciding with the multiple solutions).
Furthermore, the number of solutions to $g=0$ and to $J=0$ will have different parities,
so that one of the two functions $g$ and $J$ will change sign, guaranteeing that 
one of the stopping criteria will be triggered for such a curve segment.

Curve-tracing involves a trade-off between security and speed.
Security is enhanced with a small steplength $\Delta t$ 
and allowing only one or two Newton correction steps. 
However, these settings contribute to slow curve-tracing.  
The examples in Section~\ref{S:examples} were computed with secure settings.  
In Section~\ref{subsec:checks} we give methods to detect some errors 
in this curve-tracing which can allow less secure settings.

%
\subsection{Curve-tracing from vertices}

The curves to be traced are typically singular at the vertices of $\Delta$.
While traditional curve-tracing fails at such points, we employ a simple alternative
that takes advantage of the local structure of the curves.

At a vertex $v$ of $\Delta$, we make a linear change of coordinates so that the two
incident edges are given by the coordinate polynomials $y_1$ and $y_2$.
Then $\varphi_1(y)=0$ may be expressed as 
 \[
   \prod_{j=0}^{n+2} p_i(y)^{\beta_{i,1}}\ =\ 1\,.
 \]
Setting $p_{n+1}=y_1$ and $p_{n+2}=y_2$, we may solve for $y_2$ to obtain
\[
   y_2\ =\ y_1^{-\beta_{n+1,1}/\beta_{n+2,1}}\cdot \prod_{i=0}^n
   p_i(y)^{-\beta_{i,1}/\beta_{n+2,1}}\,.
\]
In the neighborhood of $v$ this is approximated by the monomial curve
\[
   y_2\ =\ y_1^{-\beta_{n+1,1}/\beta_{n+2,1}}\cdot \prod_{i=0}^n
   p_i(v)^{-\beta_{i,1}/\beta_{n+2,1}}\,,
\]
where we have evaluated the terms $p_i(y)$ for $i\leq n$ at the vertex $v$, 
i.e., the product is just a constant $\alpha$.
We may write this expression as
 \begin{equation}\label{Eq:monomial}
    y_2\ =\ \alpha \cdot y_1^\beta\;;
 \end{equation}
note that $\alpha>0$.

We do not begin tracing the curve $\varphi_1(y)=0$ at $v$, but instead begin from a 
point $c_p$ on the monomial curve~\eqref{Eq:monomial} in $\Delta$ near $v$.
If the first predictor-corrector step from $c_p$ succeeds to approximate $\varphi_1(y)=0$,
then we trace this curve as described before.
If this predictor-corrector step fails from $c_p$, then we simply choose a point on the
monomial curve a bit further from $v$ and try again.
In our experience, this special form of monomial tracking quickly gives way to 
usual curve-tracing on $\varphi_1(y)=0$.

%
\subsection{Procedures for checking the output}\label{subsec:checks}

General curve-tracing is not foolproof.  
If the curves to be traced are very close together, then curve-jumping may occur, leading to
missed solutions. 
However, for curve-tracing in the Khovanskii-Rolle algorithm, there is a simple way to check
for such errors. 

As described after Algorithm~\ref{Alg:Continuation_step}, each point in the sets $S_i$ for
$i=1,\dotsc,l$ should be discovered twice as the end of an arc that is tracked.
Thus, if a solution is not found twice, an error occurred in the curve-tracing.  

This check will not capture all errors.  The development of further 
verification and certification procedures is a goal for future research.  
This lack of checks is not uncommon for new algorithms, including 
the methods introduced in the early development of numerical algebraic geometry\bigskip

%
\section{Conclusions}
%

Numerical homotopy continuation is a robust and efficient algorithm 
for finding all complex solutions to a system of polynomial equations.
Real solutions are obtained by selecting those solutions with small imaginary parts.
While often practical, this is wasteful and does not exploit the real algebraic nature of the
problem. 

We presented a numerical continuation algorithm to find all real or positive
solutions to a system of polynomials, along with details of our implementation and two 
examples.  This new Khovanskii-Rolle continuation algorithm is efficient in that 
the number of paths to be followed depends upon the corresponding fewnomial bounds 
for the numbers of real solutions and not on the number of complex solutions.
This is a significant difference between our new algorithm and all other known 
methods for solving polynomial systems.

This algorithm does not directly solve the given polynomial system but rather an
equivalent (Gale dual) system consisting of master functions in the
complement of an arrangement of hyperplanes.
This appears to be the first curve-tracing algorithm that finds only
real solutions to a system of equations. 

This paper is the first step in this new line of research.
There are clear generalizations to higher dimensions.
We plan further research into techniques for 
tracing the curves that begin at singularities.
In contrast to path-following in homotopy continuation, the security of curve-tracing in this
algorithm relies on heuristics. 
Another research direction is to enhance the security and efficacy of
curve-tracing. 

When our algorithm has been implemented in more than two variables
($l>2$), we plan to use it to study real solutions to systems of polynomial equations.

 
\providecommand{\bysame}{\leavevmode\hbox to3em{\hrulefill}\thinspace}
\providecommand{\MR}{\relax\ifhmode\unskip\space\fi MR }
\providecommand{\MRhref}[2]{%
  \href{http://www.ams.org/mathscinet-getitem?mr=#1}{#2}
}
\providecommand{\href}[2]{#2}

 
\end{document}